\documentclass[
    ,final            
  ]
  {aipproc}

\layoutstyle{6x9}

\usepackage{graphicx,subfigure}
\usepackage{epstopdf}
\usepackage{amsmath,amsthm,amssymb}
\usepackage{rotating}
\usepackage{morefloats}
\usepackage{comment}

\newtheorem{thm}{Theorem}

\begin{document}

\title{Application of K\"ahler manifold to signal processing and Bayesian inference}



\author{Jaehyung Choi}{
  address={Department of Applied Mathematics and Statistics, SUNY, Stony Brook, NY 11794},
email={jj.jaehyung.choi@gmail.com}
}
\author{Andrew P. Mullhaupt}{
  address={Department of Applied Mathematics and Statistics, SUNY, Stony Brook, NY 11794},
  email={doc@zen-pharaohs.com}
}

\begin{abstract}
We review the information geometry of linear systems and its application to Bayesian inference, and the simplification available in the K\"ahler manifold case. We find conditions for the information geometry of linear systems to be K\"ahler, and the relation of the K\"ahler potential to information geometric quantities such as $\alpha $-divergence, information distance and the dual $\alpha $-connection structure. The K\"ahler structure simplifies the calculation of the metric tensor, connection, Ricci tensor and scalar curvature, and the $\alpha $-generalization of the geometric objects. The Laplace--Beltrami operator is also simplified in the K\"ahler geometry. One of the goals in information geometry is the construction of Bayesian priors outperforming the Jeffreys prior, which we use to demonstrate the utility of the K\"ahler structure.

\end{abstract}

\maketitle


\section{Introduction}

K\"ahler manifolds are important in differential geometry, with applications in several different fields such as supersymmetric gauge theory and superstring theory in theoretical physics, and in our interest, information geometry. After Barndorff-Nielsen and Jupp found the connection between statistics and symplectic geometry \cite{Barndorff-Nielsen:1997}, Barbaresco introduced K\"ahler manifolds into information geometry \cite{Barbaresco:2006} and suggested generalized complex manifolds for information geometry \cite{Barbaresco:2012,Barbaresco:2014}. Symplectic and K\"ahler structures of divergence functions are also revealed \cite{Zhang:2013}. Recently, Choi and Mullhaupt \cite{Choi:2014c} proved the mathematical correspondence between K\"ahler manifolds and the information geometry of linear systems. Moreover, the implication of the K\"ahler manifold to Bayesian inference for linear systems is also reported \cite{Choi:2014d}.

K\"ahlerian information geometry has several advantages in describing the information geometry of linear systems \cite{Choi:2014c}. First of all, geometric tensor calculation is simplified. Additionally, the $\alpha $%
-generalization of the tensors is more straightforward because the Riemann
tensor is $\alpha $-linear on the complex manifold. Moreover, searching for
the superharmonic priors suggested by Komaki \cite{Komaki:2006} is more
efficient because the Laplace--Beltrami operator in the K\"ahler geometry
is much simpler. This simplicity leads to a systematic and generic algorithm
for the geometric shrinkage priors \cite{Choi:2014d}.

In this paper, we give a review on the recent developments in the
applications of the K\"ahler manifold to information geometry, in
particular, the implications to signal processing and Bayesian inference.
First, we provide the brief fundamentals of the K\"ahler manifold, and the K\"ahlerian description of linear systems is introduced. After then, we present an application to Bayesian inference.

\section{Review on K\"ahler manifold}

We review the fundamentals of the K\"ahler manifold in this section. Let
us start with construction of a complex manifold. To extend a real manifold
to a complex manifold, the concept of complexification is necessary. Any
functions and vectors can be complexified. The complexified coordinate
system $\xi =(\xi ^{1},\cdots ,\xi ^{n})\in \mathbb{C}^{n}$ of a complex
manifold $M$ is given by 
\begin{equation}
\xi =\theta +i\zeta   \notag
\end{equation}%
where $\theta $ and $\zeta $ are the real coordinate systems of a real
manifold $N$ of dimension $n$. The complex manifold $M$ is the
complexification of the manifold $N$, denoted by $N^{\mathbb{C}}$, and it can be considered as a product manifold $N\times N$. From now on, we work
on the complex manifold $M$ of $\text{dim}_{\mathbb{C}}M=n$.

On the tangent plane at point $p$, also denoted by $T_{p}M$, the basis
vectors are given by the real coordinates as 
\begin{equation}
\{\frac{\partial }{\partial \theta ^{1}},\cdots ,\frac{\partial }{\partial
\theta ^{n}};\frac{\partial }{\partial \zeta ^{1}},\cdots ,\frac{\partial }{%
\partial \zeta ^{n}}\}  \notag
\end{equation}%
and the cotangent plane $T_{p}^{\ast }M$, which is dual to the tangent
plane, is spanned by 
\begin{equation}
\{d\theta ^{1},\cdots ,d\theta ^{n};d\zeta ^{1},\cdots ,d\zeta ^{n}\}  \notag
\end{equation}%
where $d\theta ^{i}$ and $d\zeta ^{i}$ are the one-forms of the manifold.
Since the vectors on the tangent space and the one-forms on the cotangent
space are dual to each other, the basis vectors and the one-forms should
satisfy the following identities: 
\begin{align}
\langle d\theta ^{i},\frac{\partial }{\partial \theta ^{j}}\rangle & =\delta
_{j}^{i},\langle d\theta ^{i},\frac{\partial }{\partial \zeta ^{j}}\rangle =0
\notag \\
\langle d\zeta ^{i},\frac{\partial }{\partial \theta ^{j}}\rangle &
=0,\langle d\zeta ^{i},\frac{\partial }{\partial \zeta ^{j}}\rangle =\delta
_{j}^{i}  \notag
\end{align}%
where $\langle \cdot ,\cdot \rangle $ is the inner product and $\delta
_{j}^{i}$ is the Kronecker delta.

It is also possible to describe the manifold with the complexified
coordinate system. First of all, let us introduce the following vectors:  
\begin{equation}
\frac{\partial}{\partial \xi^i}=\frac{1}{2}\Big(\frac{\partial}{\partial
\theta^i}-i \frac{\partial}{\partial \zeta^i}\Big), \frac{\partial}{\partial 
\bar{\xi}^i}=\frac{1}{2}\Big(\frac{\partial}{\partial \theta^i}+i \frac{%
\partial}{\partial \zeta^i}\Big).  \notag
\end{equation}
The tangent space $T_pM$ is spanned by the basis vectors defined above:  
\begin{equation}
\{\frac{\partial}{\partial \xi^1},\cdots,\frac{\partial}{\partial \xi^n},%
\frac{\partial}{\partial \bar{\xi}^1},\cdots,\frac{\partial}{\partial \bar{%
\xi}^n} \}  \notag
\end{equation}
and its dual cotangent space $T^*_p M$ is spanned by  
\begin{equation}
\{d\xi^1,\cdots,d\xi^n,d\bar{\xi}^1,\cdots,d\bar{\xi}^n\}  \notag
\end{equation}
such that the vectors and the one-forms in the complexified coordinate
system satisfy the similar identities in the case of the real basis vectors and one-forms:  
\begin{align}
\langle d\xi^i,\frac{\partial}{\partial \xi^j}\rangle=\delta^i_j,\langle
d\xi^i,\frac{\partial}{\partial \bar{\xi}^j}\rangle=0  \notag \\
\langle d\bar{\xi}^i,\frac{\partial}{\partial \xi^j}\rangle=0,\langle d\bar{%
\xi}^i,\frac{\partial}{\partial \bar{\xi}^j}\rangle=\delta^i_j.  \notag
\end{align}

The manifold has the almost complex structure that is the linear mapping $%
J_{p}:T_{p}M\rightarrow T_{p}M$ with 
\begin{equation}
J_{p}\frac{\partial }{\partial \xi ^{i}}=i\frac{\partial }{\partial \xi ^{i}}%
,\text{ }J_{p}\frac{\partial }{\partial \bar{\xi}^{i}}=-i\frac{\partial }{%
\partial \bar{\xi}^{i}}  \notag
\end{equation}%
and its matrix representation is the following: 
\begin{equation*}
J_{p}=\left( 
\begin{array}{cc}
i\text{ }\mathbb{I}_{n} & 0 \\ 
0 & -i\text{ }\mathbb{I}_{n}%
\end{array}%
\right) 
\end{equation*}%
where $\mathbb{I}_{n}$ the identity matrix of dimension $n$.

A Hermitian manifold is defined as a complex manifold equipped with the
metric tensor $g_p$ of the following property: 
\begin{equation}
g_{p}(J_{p}X,J_{p}Y)=g_{p}(X,Y)  \notag
\end{equation}%
where $X,Y\in T_{p}M$. First of all, the definition of the Hermitian
manifold can be represented in terms of the metric tensor components as follows: 
\begin{equation}
g_{ij}=g_{\bar{\imath}\bar{j}}=0  \label{metric_con_hermitian}
\end{equation}%
where the metric elements with the mixed indices may not vanish. In addition to that, it is always possible to construct a Hermitian manifold from any complex manifold.

One more concept for defining the K\"ahler manifold is the K\"ahler form. The K\"ahler form is defined as 
\begin{equation}
\Omega _{p}=g_{p}(J_{p}X,Y)  \notag
\end{equation}%
where $X,Y\in T_{p}M$. It is antisymmetric under the exchange of $X$ and $Y$%
: $\Omega _{p}(X,Y)=-\Omega _{p}(Y,X)$. It is expressed in terms of the metric tensor components: 
\begin{equation}
\Omega =ig_{i\bar{j}}d\xi^{i}\wedge d\bar{\xi}^{j}  \notag
\end{equation}%
where $\wedge $ is the wedge product.

Now, we are ready for defining the K\"ahler manifold. The K\"ahler manifold is defined
as the Hermitian manifold with the closed K\"ahler form. The closed K\"ahler
two-form, $d\Omega =0$, is written in the metric tensor components 
\begin{equation}
\partial _{i}g_{j\bar{k}}=\partial _{j}g_{i\bar{k}},\partial _{\bar{\imath}%
}g_{k\bar{j}}=\partial _{\bar{j}}g_{k\bar{\imath}}.
\label{metric_con_closed_kahler}
\end{equation}%
In the metric tensor expression, the geometry is K\"ahler if and only if
the metric tensor satisfies eq. (\ref{metric_con_hermitian}) and eq. (\ref%
{metric_con_closed_kahler}).

One of the most well-known properties in the K\"ahler geometry is that the metric components on the K\"ahler manifold is given by the Hessian structure:
\begin{equation}
g_{i\bar{j}}=\partial _{i}\partial _{\bar{j}}\mathcal{K}
\label{metric_kahler_rel}
\end{equation}%
where $\mathcal{K}$ is the K\"ahler potential. All the information on the
metric tensor is encoded in the K\"ahler potential. The nontrivial
elements of the Levi-Civita connection also can be expressed with the K\"a%
hler potential: 
\begin{equation}
\Gamma _{ij,\bar{k}}=\partial _{i}\partial _{j}\partial _{\bar{k}}\mathcal{K}=(\Gamma _{\bar{i}\bar{j},k})^*
\label{connection_kahler_rel}
\end{equation}%
and the other elements of the Levi-Civita connection vanish. The connection with this property is called the Hermitian connection.

Another notable fact in the K\"ahler geometry is that the Ricci tensor is
calculated from  
\begin{equation}  \label{ricci_kahler_rel}
R_{i\bar{j}}=-\partial_i\partial_{\bar{j}}\log{\mathcal{G}}
\end{equation}
where $\mathcal{G}$ is the determinant of the metric tensor. The lengthy
calculation for the Riemann curvature tensor can be skipped in the procedure of
obtaining the Ricci tensor.

Additionally, the submanifolds of the K\"ahler manifolds are also K\"ahler. If a given manifold is the K\"ahler manifold, every submanifolds are automatically K\"ahler.

Finally, it is noteworthy that the Laplace--Beltrami operator is represented
with 
\begin{equation}
\Delta =2g^{i\bar{j}}\partial _{i}\partial _{\bar{j}}  \notag
\end{equation}%
and it is much simpler than the Laplace--Beltrami operator of a non-K\"a%
hler manifold.

\section{K\"ahler geometry of signal processing}

\label{sec_maxent_kahler_signal} In this section, we cover the K\"aherian
information geometry for signal processing proposed by Choi and Mullhaupt 
\cite{Choi:2014c}. A signal filter transforms an input signal $x$ to an
output $y$ under the following linear relation: 
\begin{equation}
y(w)=h(w;\boldsymbol{\xi })x(w;\boldsymbol{\xi })  \notag
\end{equation}%
where $h(w;\boldsymbol{\xi })$ is a transfer function in frequency domain $w$%
. The $z$-transformed transfer function of a causal filter is expressed by 
\begin{equation}
h(z;\boldsymbol{\xi })=\sum_{r=0}^{\infty }h_{r}(\xi )z^{-r}  \notag
\end{equation}%
where $h_{r}$ is the $r$-th impulse response function of the linear system.
We assume that $h(z;\boldsymbol{\xi })$ is holomorphic both in $\boldsymbol{%
\xi }$ and $z$.

It is well-known by Amari and Nagaoka \cite{Amari:2000} that the metric
tensor is determined for stationary processes by the spectral density
function $S(z;\boldsymbol{\xi })=|h(z;\boldsymbol{\xi })|^{2}$. It is also
possible to write down the metric tensor in terms of the transfer function
on the complexified manifold. The metric components can be represented with $%
\eta _{r}$, the coefficient of $z^{-r}$ in the logarithmic transfer function (log-transfer function), 
\begin{align}
\label{metric_tensor_gauge1}
g_{ij}& =\partial _{i}\eta _{0}\partial _{j}\eta _{0}\\
\label{metric_tensor_gauge2} 
g_{i\bar{j}}& =\sum_{r=0}^{\infty }\partial _{i}\eta _{r}\partial _{\bar{j}}%
\bar{\eta}_{r}
\end{align}%
where $g_{\bar{\imath}\bar{j}}$ and $g_{\bar{\imath}j}$ are the complex
conjugates of $g_{ij}$ and $g_{i\bar{j}}$, respectively. It is
straightforward that $\eta _{0}=\log {h_{0}}$.

	Choi and Mullhaupt \cite{Choi:2014c} proved that the information geometry of stationary and minimum phase linear systems is K\"ahler. They also provided the conditions on the transfer
function of a linear system where the information geometry is the K\"ahler
manifold in which the Hermitian conditions, eq. (\ref{metric_con_hermitian}), are explicitly shown at  the induced metric level. In this paper, we confine ourselves to the K\"ahler manifold with explicit Hermitian metric properties, eq. (\ref{metric_con_hermitian}). In the case of a causal filter, the condition for K\"ahler manifold is as follows. 

\begin{thm}
\label{thm_tf_kahler_condition}  Given a holomorphic transfer function, the
information geometry of a signal filter is the K\"ahler manifold if and only
if $h_0$ is a constant in $\boldsymbol{\xi}$. 
\end{thm}

\begin{proof}
		If $h_0$ is a constant, the metric tensor expressions, eq.(\ref{metric_tensor_gauge1}) and eq. (\ref{metric_tensor_gauge2}), are given by
		\begin{equation}
			g_{ij}=g_{\bar{i}\bar{j}}=0, g_{i\bar{j}}=\sum_{r=1}^{\infty} \partial_i \eta_r \partial_{\bar{j}} \bar{\eta}_r\nonumber
		\end{equation}
		i.e. the manifold is Hermitian. Additionally, it is easy to check that the K\"ahler form is closed.
		
		If the geometry is K\"ahler, the manifold is Hermitian where $g_{ij}=g_{\bar{i}\bar{j}}=0$ for all $i$ and $j$. From this Hermitian condition, it is obvious that $h_0$ is constant in $\boldsymbol{\xi}$.
	\end{proof}

For the K\"ahlerian linear systems, the K\"ahler potential is the square of the Hardy norm of
the log-transfer function on the unit disk $\mathbb{D}$ \cite{Choi:2014c}:  
\begin{equation}  \label{eqn_kahler_potential_formula}
\mathcal{K}=\frac{1}{2\pi i}\oint_{|z|=1}|\log{h(z;\boldsymbol{\xi})}|^2%
\frac{dz}{z}=||\log{h(z;\boldsymbol{\xi})}||^2_{H^2}
\end{equation}
and the K\"ahler potential is also related to the 0-divergence. It is
identical to the 0-divergence for the unilateral transfer
function. It is a constant in $\alpha$ of $\alpha$-divergence.

According to the literature \cite{Choi:2014c}, the benefits of the K\"ahlerian
information geometry are the followings. First of all, the calculation of
the geometric tensors and Levi-Civita connection is simplified by the K\"ahler structure and the expressions for the geometric objects are given by eq. (\ref{metric_kahler_rel}), eq. (\ref{connection_kahler_rel}), and eq. (\ref{ricci_kahler_rel}).
Additionally, the $\alpha $-generalization of the tensors are still $\alpha $%
-linear. Finally, it is easier to find superharmonic priors on the
manifold because the Laplace--Beltrami operator on the K\"ahler manifold is in the simpler form.

We give an example: one of the most interesting linear systems is the
fractionally integrated autoregressive moving average (ARFIMA) model. For
the ARFIMA$(p,d,q)$ model of $\xi =(d,\lambda _{1},\cdots ,\lambda _{p},\mu
_{1},\cdots ,\mu _{q})$, the transfer function rescaled by the gain is given
by 
\begin{equation}
h(z;\boldsymbol{\xi })=\frac{(1-\mu _{1}z^{-1})(1-\mu _{2}z^{-1})\cdots
(1-\mu _{q}z^{-1})}{(1-\lambda _{1}z^{-1})(1-\lambda _{2}z^{-1})\cdots
(1-\lambda _{p}z^{-1})}(1-z^{-1})^{d}  \notag
\end{equation}%
where $\lambda _{i}$ is a pole from the AR part, $\mu _{i}$ is a root from
the MA part, and $d$ is a differencing parameter. The poles and the roots
are expected to be on the unit disk. By Theorem \ref{thm_tf_kahler_condition}%
, it is clear that the information geometry of the ARFIMA model is K\"a%
hler. The K\"ahler potential of the ARFIMA model, also found in the
literature \cite{Choi:2014d}, is calculated from eq. (\ref%
{eqn_kahler_potential_formula}) as 
\begin{equation}
\mathcal{K}=\sum_{k=1}^{\infty }\Big|\frac{d+(\mu _{1}^{k}+\cdots +\mu
_{q}^{k})-(\lambda _{1}^{k}+\cdots +\lambda _{p}^{k})}{k}\Big|^{2}
\label{eqn_kahler_potential_ARFIMA}
\end{equation}%
and it is bounded above by $(d+p+q)^{2}\frac{\pi ^{2}}{6}$. The metric
tensor, derived from eq. (\ref{metric_kahler_rel}), is represented by 
\begin{equation*}
g_{i\bar{j}}=\left( 
\begin{array}{ccc}
\frac{\pi ^{2}}{6} & \frac{1}{\bar{\lambda}_{j}}\log {(1-\bar{\lambda}_{j})}
& -\frac{1}{\bar{\mu}_{j}}\log {(1-\bar{\mu}_{j})} \\ 
\frac{1}{\lambda _{i}}\log {(1-\lambda _{i})} & \frac{1}{1-\lambda _{i}\bar{%
\lambda}_{j}} & -\frac{1}{1-\lambda _{i}\bar{\mu}_{j}} \\ 
-\frac{1}{\mu _{i}}\log {(1-\mu _{i})} & -\frac{1}{1-\mu _{i}\bar{\lambda}%
_{j}} & \frac{1}{1-\mu _{i}\bar{\mu}_{j}}%
\end{array}%
\right) 
\end{equation*}%
where the first column and the first row are for the direction of the
fractional differencing parameter $d$. It is easy to find the metric tensor
for the pure ARMA model as a submanifold of the ARFIMA geometry.

The non-trivial connection elements are also found from eq. (\ref{connection_kahler_rel}) and it
is noteworthy that the connection components with the differencing parameter
direction at any index of the first two indices are all vanishing. The Ricci tensor components are also calculated by eq. (\ref%
{ricci_kahler_rel}) and it is also vanishing along the $d$-direction. The non-vanishing Ricci tensor components are only from the pure ARMA directions with the
correction term from the mixing between the pure ARMA piece and the
fractionally integrated part:  
\begin{equation}
R_{i\bar{j}}=R_{i\bar{j}}^{ARMA}+ R_{i\bar{j}}^{ARMA-FI}  \notag
\end{equation}
where $i$ and $j$ are not along the $d$-direction.

\section{Geometric shrinkage priors of K\"ahlerian filters}

\label{sec_maxent_geo_prior} First, we review the superharmonic priors
proposed by Komaki \cite{Komaki:2006}. The difference in risk function
between two Bayesian predictive densities from the Jeffreys prior $\pi _{J}$
and the superharmonic prior $\pi _{I}$ with respect to the density $p(y|%
\boldsymbol{\xi })$ is given by 
\begin{align}
& \mathbb{E}[D_{KL}(p(y|\boldsymbol{\xi })||p_{\pi _{J}}(y|x^{(N)}))|%
\boldsymbol{\xi }]-\mathbb{E}[D_{KL}(p(y|\boldsymbol{\xi })||p_{\pi
_{I}}(y|x^{(N)}))|\boldsymbol{\xi }]  \notag \\
& =\frac{1}{2N^{2}}g^{ij}\partial _{i}\log {\Big(\frac{\pi _{I}}{\pi _{J}}%
\Big)}\partial _{j}\log {\Big(\frac{\pi _{I}}{\pi _{J}}\Big)}-\frac{1}{N^{2}}%
\frac{\pi _{J}}{\pi _{I}}\Delta \Big(\frac{\pi _{I}}{\pi _{J}}\Big)+o(N^{-2})
\notag
\end{align}%
where $N$ is the size of samples $x$. If a positive prior function $\psi =\pi
_{I}/\pi _{J}$ is superharmonic, the risk function of the Bayesian
predictive density $p_{\pi _{I}}$ is decreased with respect to that of $p_{\pi _{J}}$%
, the predictive density from the Jeffreys prior. Comparing with $p_{\pi
_{J}}$, $p_{\pi _{I}}$ is closer to $p(y|\boldsymbol{\xi })$ in the
Kullback-Leibler divergence. Superharmonic priors for several probability
distributions and linear systems have been found \cite{Komaki:2006, Tanaka:2006,
Tanaka:2009, Choi:2014c, Choi:2014d}.

A difficulty in Komaki's idea is that it is non-trivial to test the
superharmonicity for a prior function $\psi $ of a general statistical model or a linear system with high dimensionality. Although a superharmonic prior
for the AR model in an arbitrary dimension was found by Tanaka \cite%
{Tanaka:2009}, no superharmonic priors for the ARMA models and the ARFIMA
models were found. Moreover, any systematic algorithms for finding the Komaki priors were not known.

Recently, a generic algorithm for the shrinkage priors of linear systems is
introduced when the information geometry of the model is K\"ahler \cite%
{Choi:2014d}. Superharmonic priors for more general
time series models and signal filters are efficiently constructed by the
algorithm. The following theorem is useful to find the superharmonic prior
functions.

\begin{thm}
\label{thm_algo_prior} On a K\"ahler manifold, a positive function $\psi
=\Psi (u^{\ast }-\kappa (\boldsymbol{\xi },\bar{\boldsymbol{\xi }}))$ is a
superharmonic prior function if $\kappa (\boldsymbol{\xi },\bar{\boldsymbol{%
\xi }})$ is subharmonic (or harmonic), bounded above by $u^{\ast }$, and $\Psi $ is
concave decreasing: $\Psi ^{\prime }(\tau )>0$, $\Psi ^{\prime \prime }(\tau
)\le 0$  (or $\Psi ^{\prime }(\tau )>0$, $\Psi ^{\prime \prime }(\tau
)< 0$). 
\end{thm}

\begin{proof}
		The proof is given in the literature \cite{Choi:2014d}.
	\end{proof}

If we find a positive subharmonic or harmonic function, we apply Theorem \ref%
{thm_algo_prior} to obtain a superharmonic function and exploit the
superharmonic function as a shrinkage prior function for prediction as
Komaki \cite{Komaki:2006} suggested.

Fortunately, several choices for $\Psi $ and $\kappa $ are already known \cite%
{Choi:2014d}. The candidates for $\Psi $ are the followings:%
\begin{align}
\Psi _{1}(\tau )& =\tau ^{a}  \notag \\
\Psi _{2}(\tau )& =\log {(1+\tau ^{a})}  \notag
\end{align}%
where $\tau $ is positive and $0<a\leq 1$ for subharmonic $\kappa$ (or $0<a< 1$ for harmonic $\kappa$). Moreover, the ans\"atze for $%
\kappa $ are found as 
\begin{align}
\kappa _{1}& =\mathcal{K}  \notag \\
\kappa _{2}& =\sum_{r=0}^{\infty }a_{r}|h_{r}(\boldsymbol{\xi })|^{2}  \notag
\\
\kappa _{3}& =\sum_{i=1}^{n}b_{i}|\xi ^{i}|^{2}  \notag
\end{align}%
where $a_{r}$ and $b_{i}$ are positive real numbers. In particular, $\kappa
_{1}$ is the K\"ahler potential which is intrinsic on the K\"ahler
manifold. By combining $\kappa $ and $\Psi $, it is easy to construct
geometric shrinkage priors like 
\begin{align}
\psi _{1}& =(u^{\ast }-\mathcal{K})^{a}  \notag \\
\psi _{2}& =\log {(1+(u^{\ast }-\mathcal{K})^{a})}  \notag
\end{align}%
which outperform the Jeffreys prior in the viewpoint of information
theory.

\section{Conclusion}

We reviewed information geometric applications of K\"ahler manifolds
to linear systems and Bayesian inference, and exposed that the simpler Laplace--Beltrami operator, one of the advantages in the K\"ahlerian approach, is applicable to Bayesian inference:
finding superharmonic priors on the K\"ahler manifold is
straightforward, as we have shown for linear systems, in particular, the ARFIMA models.

\begin{theacknowledgments}
	We are grateful to Fr\'ed\'eric Barbaresco, Robert J. Frey, Hiroshi Matsuzoe, Michael Tiano, and Jun Zhang for useful discussions. We thank Fr\'ed\'eric Barbaresco for notifying his notable works on the K\"ahler geometry and information geometry. We are also thankful to the participants and the organizers of MaxEnt 2014 in Amboise, France.
\end{theacknowledgments}




\bibliographystyle{aipproc}
\bibliography{sample}

\IfFileExists{\jobname.bbl}{}  {\typeout{}  %
\typeout{******************************************}  \typeout{** Please run
"bibtex \jobname" to optain}  \typeout{** the bibliography and then re-run
LaTeX}  \typeout{** twice to fix the references!}  %
\typeout{******************************************}  \typeout{}  }

\end{document}